\documentclass[12pt]{amsart}
\usepackage[top=3.5cm, bottom=3.5cm, left=3.8cm, right=3.8cm]{geometry}
\usepackage{graphicx} 
\usepackage{amsmath}
\usepackage{amsthm}
\usepackage{amssymb}
\usepackage{graphicx}
\usepackage{enumitem}
\newtheorem{theorem}{Theorem}[section]

\newtheorem{lemma}{Lemma}[section]
\theoremstyle{definition}

\newtheorem{remark}{Remark}

\numberwithin{equation}{section}

\begin{document}

\long\def\symbolfootnote[#1]#2{\begingroup%
\def\thefootnote{\fnsymbol{footnote}}\footnote[#1]{#2}\endgroup} 
\title{\textbf{A note on composition operators on the disc and bidisc.}}
\author{Athanasios Beslikas}
\maketitle
\begin{abstract} 
In this note we give a new necessary condition for the boundedness of the composition operator on the Dirichlet-type space on the disc, via a two dimensional change of variables formula. With the same formula, we characterise the bounded composition operators on the anisotropic Dirichlet-type spaces $\mathfrak{D}_{\vec{a}}(\mathbb{D}^2)$ induced by holomorphic self maps of the bidisc $\mathbb{D}^2$ of the form $\Phi(z_1,z_2)=(\phi_1(z_1),\phi_2(z_2))$. We also consider the problem of boundedness of  composition operators $C_{\Phi}:\mathfrak{D}(\mathbb{D}^2)\to A^2(\mathbb{D}^2)$ for general self maps of the bidisc, applying some recent results about Carleson measures on the the Dirichlet space of the bidisc.  
\end{abstract}
 \section{Introduction} In the present work we will study the composition operators in the Dirichlet-type spaces on the setting of the unit disc and the bidisc. Moreover, the problem of boundedness of the composition operator between the Bergman and Dirichlet spaces on the bidisc $\mathbb{D}^2$ is studied. A standard reference for composition operators in the Hardy and Bergman space is the celebrated work of Shapiro \cite{Shapiro}. Another very influential work is the article of Pau and Perez, in which the authors provide criteria for the boundedness, compactness and closed range composition operators on the Dirichlet-type space on the unit disc, mainly using the \textit{generalised Nevanlinna counting function} and the pull-back measure. In this note we will provide a new sufficient condition for the boundedness of the composition operator acting on the Dirichlet space, which utilises the \textit{generalised Nevanlinna counting function}, along with some other novel techniques. More precisely, we will follow the approach of Pau and Perez on the bidisc setting, for symbols $\Phi=(\varphi(z_1),\varphi(z_2)),$ and obtain a characterization for boundedness of the composition operator on the bidisc via the \textit{generalized Nevanlinna counting function}. Also, we will provide some insight on the problem of boundedness of the composition operator $C_{\Phi}:\mathfrak{D}(\mathbb{D}^2)\to A^2(\mathbb{D}^2)$ using known results for Carleson measures on the Dirichlet space of the bidisc, as well as new results on the matter.\\The problem of boundedness of the composition operator on the bidisc and tridisc for the Hardy and Bergman spaces, has been treated by Kosinski and Bayart in their papers \cite{Kosinski}, \cite{Bayart1}, \cite{Bayart2} respectively, for symbols $\Phi$ who are $\mathcal{C}^2$ smooth on the closed bidisc $\overline{\mathbb{D}^2}.$ In our note, we establish a connection of the above mentioned problem with the Carleson measures for Dirichlet-type spaces on the bidisc.  Let us now give some essential notation. For the whole length of the paper, we denote by $\mathbb{D}$ the unit disc on the one-dimensional complex plane, and by $\varphi:\mathbb{D}\to\mathbb{D}$ a holomorphic self-map of the disc. For a holomorphic function $f$ on the disc or the bidisc, we will write $f\in H(\mathbb{D})$ or $f\in\mathbb{D}^2$ respectively. By 
 $\mathbb{D}^2=\{z\in \mathbb{C}^2: |z_1|<1,|z_2|<1\},$ we denote the bidisc in the two-dimensional complex plane, while by
 $\Phi:\mathbb{D}^2\to \mathbb{D}^2$ we denote a holomorphic self-map of the unit bidisc. Consider now the composition operators $$C_{\varphi}f=f\circ \varphi(z),z\in\mathbb{D}$$ and
 $$C_{\Phi}f=f\circ\Phi(z_1,z_2)=f(\varphi_1(z_1,z_2),\varphi_2(z_1,z_2)),$$
where $z_1,z_2\in \mathbb{D}$ and both $\varphi_1,\varphi_2$ map the bidisc to the unit disc. Of course, if $f$ is a holomorphic function in one, or two variables, then the composition operator induces a holomorphic function on the disc or the bidisc respectively. 
\subsection{Plan of the paper} The presentation of the results of this note will go as follows: In Section 2 we prove a two-dimensional change of variables for \textit{ separated symbols}. We apply this useful tool to prove the two main results of this section, namely, a new necessary condition for boundedness of the composition operator acting on the Dirichlet space of the disc, and a characterization of the holomorphic self-maps of the bidisc of the form $$\Phi=(\varphi_1(z_1),\varphi_2(z_2)).$$ Note that these symbols have one-dimensional functions as co-ordinate functions. By following the standard techniques of \cite{Pau} we will prove a characterization of the boundedness of the composition operator acting on the anisotropic Dirichlet space $\mathfrak{D}_{\vec{a}}(\mathbb{D}^2)$ for these symbols.\\
 In Section 3, we give a sufficient condition for the composition operator  $C_{\Phi}:\mathfrak{D}(\mathbb{D}^2)\to A^2(\mathbb{D}^2)$ to be bounded, by proving a sufficient condition for a positive Borel measure to be Carleson measure on the Dirichlet space of the bidisc.  We will also prove a necessary and sufficient condition for the boundedness of the composition operator $C_{\Phi}:\mathfrak{D}(\mathbb{D}^2)\to A^2(\mathbb{D}^2)$ , applying a very important result of N.Arcozzi, P.Mozolyako, K-M.Perfekt and G.Sarfatti, in their seminal work \cite{Arcozzi2}
 which characterises the Carleson measures in the Dirichlet space of the bi-disc. The main result of this section can be stated only for the unweighted Dirichlet space. The condition we obtain is necessary and sufficient for the composition operator $C_{\Phi}:\mathfrak{D}(\mathbb{D}^2)\to A^2(\mathbb{D}^2)$ to be bounded but it is not as pleasant as one could wish because it involves a union of Carleson boxes and not one box as in Theorem 10 of \cite{Kosinski}.
 \subsection{Definitions and tools}  In this first section we will present some backround material, essential for the interested reader. First and foremost, we recall that the Dirichlet-type space $\mathfrak{D}_{a}(\mathbb{D})$ is consisted of all holomorphic functions on the disc, which satisfy
 $$\int_{\mathbb{D}}|f'(z)|^2(1-|z|^2)^{a}dA(z)<+\infty $$ for $0<a<1.$ These spaces have been studied extensively in several works, one of which is \cite{ransf}. The problem of boundedness and compactness of the composition operator in these spaces has been treated in \cite{Pau}.
 Now we pass to the setting of the bi-disc. The anisotropic Dirichlet-type spaces on the bidisc $\mathbb{D}^2$, denoted by $\mathfrak{D}_{\vec{a}}(\mathbb{D}^2)$, where $\vec{a}=(a_1,a_2)\in\mathbb{R}^2,$ are defined as follows: Firstly, we remind the interested reader that if $f\in H(\mathbb{D}^2)$ then $f$ has a series expansion as follows:
 $$f(z_1,z_2)=\sum_{k\ge 0} \sum_{k\ge 0}a_{k,l}z_1^kz_2^l.$$
 We say that a holomorphic function $f$ on the bidisc, belongs to the Dirichlet type spaces if
 $$||f||^2_{\mathfrak{D}_{\vec{a}}(\mathbb{D}^2)}=\sum_{k\ge 0} \sum_{k \ge 0}(k+1)^{2a_1}(l+1)^{2a_2}|a_{k,l}|^2<\infty. $$
 For $0<a_1,a_2\leq\frac{1}{2}$ we can obtain an equivalent integral representation of the norm of a function in the Dirichlet-type spaces. Specifically,
 $$||f||^2_{\mathfrak{D}_{\vec{a}}(\mathbb{D}^2)}=|f(0,0)|^2+\mathfrak{D}_{\vec{a}}(f),$$
 where, 
 \begin{multline}
 \mathfrak{D}_{\vec{a}}(f)=\int_{\mathbb{D}}|\partial_{z_1}f(z_1,0)|^2dA_{a_1}(z_1)+\int_{\mathbb{D}}|\partial_{z_2}f(0,z_2)|^2dA_{a_2}(z_2)+\\ \int_{\mathbb{D}^2}|\partial_{z_2}\partial_{z_1}f(z_1,z_2)|^2dA_{a_1}(z_1)dA_{a_2}(z_2), i=1,2.
 \end{multline}
where $$dA_{a_i}(z_k)=(1-|z_k|^2)^{1-2a_i}dA(z_k),$$ and $dA(z_k)$ is the Lebesgue measure on the disc, that is $dA(z_k)=\frac{1}{\pi}dx_kdy_k.$
For the case where $a_1=a_2=1/2$ we obtain the classical Dirichlet space of the bidisc, as it was defined in \cite{Kaptanoglu} by Kaptanoglu.
 It is an easy fact to check that these spaces are Hilbert spaces with respect to the norm defined above, with the usual, naturally induced inner product. Also, they are invariant under unitary transformations but not invariant under automorphisms of the bidisc, as it is showcased in \cite{Kaptanoglu}. These spaces enjoy the property of bounded point evaluation functional. For more on the basic theory of the Dirichlet type spaces on the bidisc one can study \cite{Kaptanoglu}. Furthermore, in the seminal work of Knese, Kosinski, Ransford and Sola,  \cite{Kosinski2} one can find characterisations about the cyclic polynomials on the above mentioned spaces. In the present note we introduce the problem of the boundedness of the composition operator. Until now, no work has been found in the literature regarding this specific problem. As an initial step, we characterize the symbols $\Phi$ of the form $\Phi(z_1,z_2)=(\varphi_1(z_1),\varphi_2(z_2)),z_1,z_2\in\mathbb{D}$ (each $\varphi_i,i=1,2$ is a holomorphic self map of the unit disc $\mathbb{D})$ that induce bounded composition operator on $\mathfrak{D}_{\vec{a}}(\mathbb{D}^2).$ For these symbols the composition operator $C_{\Phi}:\mathfrak{D}_{\vec{a}}(\mathbb{D}^2)\to \mathfrak{D}_{\vec{a}}(\mathbb{D}^2)$ is easier to tame. In the more general case, where both co-ordinate functions are mapping the bidisc to the disc, things are more difficult.\\ The general case, due to the insurmountable amount of calculations seems much more difficult to handle. To explain further, the techniques which allowed Bayart and Kosinski to prove the main results of their papers was the known characterization of the Carleson measures. Taking the corresponding pull-back measure provided them with the basic tool. On the Dirichlet-type spaces case, for general symbols, this is not easy to obtain, as there are derivations which significantly change the nature of the pull-back measure. Instead of the general case, we will consider an alternative problem, and that is of the boundedness of the composition operator between the classical Bergman space $A^2_{\beta}(\mathbb{D}^2),\beta>-1,$ which is consisted of the holomorphic functions in $\mathbb{D}^2$ which satisfy $$\int_{\mathbb{D}^2}|f(z_1,z_2)|^2dV_{\beta}(z_1,z_2)<\infty,$$ where $dV_{\beta}(z_1,z_2)=dA_{\beta}(z_1)\times dA_{\beta}(z_2),$ and the Dirichlet space on the bidisc that we briefly discussed before. One can realise the difficulty and the technicality of the problem by studying the recent works of Bayart \cite{Bayart1} and Kosinski \cite{Kosinski} respectively, for the Hardy and Bergman spaces of the bidisc and tridisc respectively. 
 \section{A necessary condition and the case of separated symbols} Throughout this section we consider symbols $\Phi$ of the form
 $$\Phi(z_1,z_2)=(\varphi_1(z_1),\varphi_2(z_2)),z_1,z_2\in \mathbb{D}$$
 where both $\varphi_1,\varphi_2$ are holomorphic self maps of the unit disc.
  We shall call these symbols "separated symbols" for convenience. In the sequel we will first prove a new sufficient condition for the boundedness of the composition operator acting on $\mathfrak{D}_a(\mathbb{D}).$ In order to establish our results, a useful tool will be the \textit{Generalised Nevanlinna Counting Function} 
  $$\mathcal{N}_{\varphi,\alpha}(z)=\sum_{\varphi(w)=z}\left(\log\frac{1}{|w|}\right)^{1-2a},z\in \mathbb{D}\setminus\{\varphi(0)\},$$ where $\varphi$ is a holomorphic self map of the disc. We will also need a recent characterization of the Dirichlet-type spaces, specifically the following result of Balooch and Wu \cite{Balooch}.
   \begin{theorem}
   \textit{Let $\sigma, \tau \ge -1 $ and $\beta\in \mathbb{R},$ such that $\frac{\max(\sigma,\tau)}{2}-1<\beta\leq \frac{\sigma+\tau}{2}.$ Let $f\in H(\mathbb{D}).$ Then:}
   $$\int_{\mathbb{D}}\int_{\mathbb{D}}\frac{|f(z)-f(w)|^2}{|1-\overline{w}z|^{2(\beta+2)}}dA_{\sigma}(z)dA_{\tau}(w)\asymp ||f||^2_{\mathcal{D}_{\sigma+\tau-2\beta}(\mathbb{D})}.$$
   \end{theorem}
   The notation "$\asymp$" means that the quantities on both sides are equivalent. For what follows, we denote by $$k^{\varphi}(z_1,z_2)=\frac{1-\varphi(z_1)\overline{\varphi(z_2)}}{1-z_1\overline{z_2}}.$$ The main results of this section are the following:
   \begin{theorem} \textit{Let $\varphi:\mathbb{D}\to\mathbb{D}$ be a holomorphic self map of the disc. Assume that $C_{\varphi}:\mathfrak{D}_a(\mathbb{D})\to \mathfrak{D}_a(\mathbb{D})$ is bounded, for $a=2\sigma-2\beta, \sigma-1<\beta\leq \sigma.$ Then:}
   \begin{enumerate}
   \item $$\sup_{z_1,z_2\in\mathbb{D}}\frac{|k^{\varphi}(z_1,z_2)|}{|\varphi'(z_1)\cdot \varphi'(z_2)|^{\frac{1}{\beta+2}}}<+\infty$$
   \item $$||C_{\varphi}||_{\mathrm{operator}}\leq C \left(\sup_{z_1,z_2\in\mathbb{D}}\frac{|k^{\varphi}(z_1,z_2)|}{|\varphi'(z_1)\cdot \varphi'(z_2)|^{\frac{1}{\beta+2}}}\right)^{\beta+2},$$
   \textit{where $||\cdot||_{\mathrm{operator}}$ denotes the operator norm.}
   \end{enumerate}
   \end{theorem}
The next result can also be seen as an application of the change of variables formula, and is the following:
 \begin{theorem} \textit{Let $\Phi:\mathbb{D}^2\to \mathbb{D}^2$ be a seperated, holomorphic self map of the bidisc. Then, the composition operator $C_{\Phi}:\mathfrak{D}_{\vec{a}}(\mathbb{D}^2)\to\mathfrak{D}_{\vec{a}}(\mathbb{D}^2)$} is bounded if and only if:
 $$\sup_{z_i\in\mathbb{D}}\frac{\mathcal{N}_{\varphi_i,a_i}(z_i)}{(1-|z_i|^2)^{1-2a_i}}<\infty, i=1,2, 0<a_1,a_2<1/2$$
 \end{theorem}
 First we will prove Theorem 3.2. In order to be able to showcase the proof, we require only the change-of-variable formula that we mentioned in the introductory section of the paper.
 \begin{lemma} \textit{Let $g:\mathbb{D}^2\to \mathbb{C}$ be a positive integrable function, and $\varphi_1,\varphi_2:\mathbb{D}\to\mathbb{D}$ holomorphic self maps of the disc. Then:}
 \begin{multline}
 \int_{\mathbb{D}^2}g(\varphi_{1}(z_1),\varphi_{2}(z_2))|\varphi_1'|^2|\varphi_2'|^2dA_{a_1}(z_1)dA_{a_2}(z_2)=\\=\int_{\mathbb{D}^2}g(w_1,w_2)\mathcal{N}_{\varphi_1,a_1}(w_1)\mathcal{N}_{\varphi_{2},a_2}(w_2)dV(w_1,w_2)
 \end{multline}
 \end{lemma}
 \begin{proof} Fix $z_2\in \mathbb{D}.$ Then the slice function $g_{z_2}(z_1)$ is integrable and positive in $\mathbb{D},$ hence the change of variable formula holds for the disc and for $w_1=\varphi(z_1).$ That is:
$$\int_{\mathbb{D}}g(\varphi(z_1),\varphi_2(z_2))|\varphi'_1|^2dA_{a_1}(z)=\int_{\mathbb{D}}g(w_1,\varphi_2(z_2))\mathcal{N}_{\varphi_1,a_1}(w_1)dA(w_1).$$
Next, we multiply both sides of the above equality by $$(1-|z_2|^2)^{1-2a_2}|\phi'_2(z_2)|^2$$ and we integrate both sides into polar rectangles $R_k.$
We obtain: 
\begin{multline}
\int_{R_k}(1-|z_2|^2)^{1-2a_2}|\phi'_2|^2\left(\int_{\mathbb{D}}g(\varphi(z_1),\varphi_2(z_2))|\varphi'_1|^2dA_{a_1}(z_1)\right)dA(z_2)=\\ \int_{\mathbb{D}}\mathcal{X}_{R_k}(z_2)(1-|z_2|^2)^{1-2a_2}|\phi'_2|^2\left(\int_{\mathbb{D}}g(w_1,\varphi_2(z_2))\mathcal{N}_{\varphi_1,a_1}(w_1)dA(w_1)\right)dA(z_2),
\end{multline}
where $\mathcal{X}_{R_k}$ denotes the characteristic function of each polar rectangle $R_k.$
Summing both sides in terms of $k$ and applying Fubini's Theorem, we obtain the result.
 \end{proof}
At this point we are ready to proceed with the proof of Theorem 3.2.
\begin{proof}
Pick $\beta,\sigma$ as in the statement of the Theorem. Set $$g(z_1,z_2)=\frac{|f(z_1)-f(z_2)|^2}{|1-\overline{z_2}z_1|^{2\beta+2}},z_1,z_2\in\mathbb{D},$$
which is a positive integrable function on the bi-disc. We receive:
\begin{multline}
\int_{\mathbb{D}^2}\frac{|f\circ \varphi(z_1)-f\circ \varphi(z_2)|^2}{|1-\overline{\varphi(z_2)}\varphi(z_1)|^{2(\beta+2)}}|\varphi'(z_1)|^2|\varphi'(z_2)|^2dV_{\sigma}(z_1,z_2)=\\ =\int_{\mathbb{D}^2}\frac{|f(w_1)-f(w_2)|^2}{|1-\overline{w_2}w_1|^{2(\beta+2)}}\mathcal{N}_{\varphi,\sigma}(w_1)\mathcal{N}_{\varphi,\sigma}(w_2)dV(w_1,w_2)
\leq C||f||^2_{\mathfrak{D}_{p}(\mathbb{D})}
\end{multline}
with the last inequality obtained by our assumption about the boundedness of the operator (Theorem 3.1. in \cite{Pau}) along with Theorem 3.1.
Now we multiply and divide with $|1-\overline{z_2}z_1|^{2(\beta+2)}$ both the numerator and the denominator inside the left hand side integral and we obtain:
$$\int_{\mathbb{D}^2}\frac{|f\circ \varphi(z_1)-f\circ \varphi(z_2)|^2}{|1-\overline{z_2}z_1|^{2(\beta+2)}}\frac{|\varphi'(z_1)|^2|\varphi'(z_2)|^2}{|k^{\varphi}(z_1,z_2)|^{2(\beta+2)}}dV_{\sigma}(z_1,z_2)\leq C ||f||^2_{\mathfrak{D}_{p}(\mathbb{D})}.$$
Taking the infinum out of the integral, yields:
$$\inf_{z_1,z_2\in\mathbb{D}}\frac{|\varphi'(z_1)|^2|\varphi'(z_2)|^2}{|k^{\varphi}(z_1,z_2)|^{2(\beta+2)}}\int_{\mathbb{D}^2}\frac{|f\circ \varphi(z_1)-f\circ \varphi(z_2)|^2}{|1-\overline{z_2}z_1|^{2(\beta+2)}}dV_{\sigma}(z_1,z_2)\leq C ||f||^2_{\mathfrak{D}_{p}(\mathbb{D})}$$
which, by Theorem 3.1. and the properties of supremum, is equivalent to:
$$||C_{\varphi}f||^2_{\mathfrak{D}_{p}(\mathbb{D})}\leq C\sup_{z_1,z_2\in\mathbb{D}}\frac{|k^{\varphi}(z_1,z_2)|^{2(\beta+2)}}{|\varphi'(z_1)|^2|\varphi'(z_2)|^2} ||f||^2_{\mathfrak{D}_{p}(\mathbb{D})}.$$
Because of our assumption that the composition operator is bounded for all holomorphic functions $f,$ the supremum has to be bounded. This provides us with the result. For the second part, we just observe that
$$||C_{\varphi}||_{\mathrm{operator}}=\inf\{c\ge0: ||C_\varphi||_{\mathfrak{D}_{p}}\leq c||f||_{\mathfrak{D}_{p}}\}\leq C \left(\sup_{z_1,z_2\in\mathbb{D}}\frac{|k^{\varphi}(z_1,z_2)|}{|\varphi'(z_1)\cdot \varphi'(z_2)|^{\frac{1}{\beta+2}}}\right)^{\beta+2}$$
which comes straightforwardly by the inequality above.
\end{proof}
Before we are able to prove the second result, we have to give an auxillary lemma . We state the following lemma for the composition $f\circ \Phi,$ where $\Phi=(\varphi_1(z_1),\varphi_2(z_2)),z_1,z_2\in\mathbb{D},$ is a \textit{separated symbol.}
\begin{lemma}
\textit{Let $\Phi:\mathbb{D}^2\to \mathbb{D}^2$ be a separated symbol. Then the following holds:}
\begin{multline}
 ||C_{\Phi}f||^2_{\mathfrak{D}_{\vec{a}}(\mathbb{D}^2)}=\int_{\mathbb{D}}|\partial_{z_1}f(z_1,0)|^2\mathcal{N}_{\varphi_1,a_1}(z_1)dA(z_1)+\\\int_{\mathbb{D}}|\partial_{z_2}f(0,z_2)|^2\mathcal{N}_{\varphi_2,a_2}(z_2)dA(z_2)+\\ \int_{\mathbb{D}^2}|\partial_{z_2}\partial_{z_1}f(z_1,z_2)|^2\mathcal{N}_{\varphi_1,a_1}(z_1)\mathcal{N}_{\varphi_2,a_2}(z_2)dA(z_1)dA(z_2).
 \end{multline}
\end{lemma}
\begin{proof}  We apply the one-dimensional change of variables formula for the first two integrals defining the norm, and for the third integral we use the two-dimensional version of the change of variables we proved before.  For the 1st term in the integral norm of the composition operator  we just apply the well known change-of-variables formula for the function $g(z_1)=|\partial_{z_1}f(z_1,0)|^2$, and similarly for the second term, we apply it for the function $g(z_2)=|\partial_{z_2}f(0,z_2)|^2.$ For the third term we apply Lemma 3.2 for the function $g(z_1,z_2)=|\partial_{z_2}\partial_{z_1}f(z_1,z_2)|^2.$
\end{proof}
At this point we are ready for the proof of Theorem 3.3.
\begin{proof} (Proof of Theorem 3.3) For the sufficiency part, assume that there exist two positive constants $C_1,C_2$ such that:
$$\mathcal{N}_{\varphi_1,a_1}(z_1)\leq C_1(1-|z_1|^2)^{1-2a_1}$$ and
$$\mathcal{N}_{\varphi_2,a_2}(z_2)\leq C_2(1-|z_2|^2)^{1-2a_2}$$
then the following three inequalities hold:
$$\int_{\mathbb{D}}|\partial_{z_1}f(z_1,0)|^2\mathcal{N}_{\varphi_1,a_1}(z_1)dA(z_1)\leq C_1\int_{\mathbb{D}}|\partial_{z_1}f(z_1,0)|^2dA_{a_1}(z_1),$$
$$\int_{\mathbb{D}}|\partial_{z_2}f(0,z_2)|^2\mathcal{N}_{\varphi_2,a_2}(z_2)dA(z_2)\leq C_2\int_{\mathbb{D}}|\partial_{z_2}f(0,z_2)|^2dA_{a_2}(z_2),$$
and
\begin{multline}\int_{\mathbb{D}^2}|\partial_{z_2}\partial z_1f(z_1,z_2)|^2\mathcal{N}_{\varphi_1,a_1}(z_1)\mathcal{N}_{\varphi_2,a_2}(z_2)dA(z_1)dA(z_2)\\\leq C_1C_2\int_{\mathbb{D}^2}|\partial_{z_1}\partial_{z_2}f(z_1,z_2)|^2dA_{a_1}(z_1)dA_{a_2}(z_2).
\end{multline}
Moreover, by the boundedness of the point evaluation functional on $\mathfrak{D}_{\vec{a}}(\mathbb{D}^2)$ we obtain a positive constant $C_3$ such that $$|f(\varphi_1(0),\varphi_2(0))|\leq C_3|f(0,0)|.$$
Hence, we receive:
$$||C_{\Phi}f||_{\mathfrak{D}_{\vec{a}}(\mathbb{D}^2)}\leq C||f||_{\mathfrak{D}_{\vec{a}}(\mathbb{D}^2)} $$
and this direction of the proof is finished.
For the necessity part we follow the proof of Theorem 3.1. in \cite{Pau}. We will write down the proof for convenience of the reader. We assume now that there exists a constant $C>0$ such that:
$$||C_{\Phi}f||_{\mathfrak{D}_{\vec{a}}(\mathbb{D}^2)}\leq C||f||_{\mathfrak{D}_{\vec{a}}(\mathbb{D}^2)},$$
for all $f\in\mathfrak{D}_{\vec{a}}(\mathbb{D}^2).$
We consider the following test function:
\begin{multline}
f_{\omega_1,\omega_2}(z_1,z_2)=(1-|\omega_1|^2)^{1+\frac{1-2a_1}{2}}\int_{0}^{z_1}\frac{d\zeta_1}{(1-\overline{\omega_1}\zeta_1)^{3-2a_1}}\\
+(1-|\omega_2|^2)^{1+\frac{1-2a_2}{2}}\int_{0}^{z_2}\frac{d\zeta_2}{(1-\overline{\omega_2}\zeta_2)^{3-2a_2}}\\
+(1-|\omega_1|^2)^{1+\frac{1-2a_1}{2}}(1-|\omega_2|^2)^{1+\frac{1-2a_2}{2}}\int_{0}^{z_1}\int_{0}^{z_2}\frac{d\zeta_1d\zeta_2}{(1-\overline{\omega_1}\zeta_1)^{3-2a_1}(1-\overline{\omega_2}\zeta_2)^{3-2a_2}},
\end{multline}
where $\omega_1,\omega_2 \in \mathbb{D}.$
By calculating the partial derivatives:
\begin{multline} ||C_{\Phi}f_{\omega_1,\omega_2}||^2_{\mathfrak{D}_{\vec{a}}(\mathbb{D}^2)}=\int_{\mathbb{D}}\frac{(1-|\omega_1|^2)^{3-2a_1}}{|1-\overline{\omega_1}z_1|^{6-4a_1}}\mathcal{N}_{\varphi_1,a_1}(z_1)dA(z_1)+\\
\int_{\mathbb{D}}\frac{(1-|\omega_2|^2)^{3-2a_2}}{|1-\overline{\omega_2}z_2|^{6-4a_2}}\mathcal{N}_{\varphi_2,a_2}(z_2)dA(z_2)+\\
\int_{\mathbb{D}}\int_{\mathbb{D}}\frac{(1-|\omega_1|^2)^{3-2a_1}}{|1-\overline{\omega_1}z_1|^{6-4a_1}}\frac{(1-|\omega_2|^2)^{3-2a_2}}{|1-\overline{\omega_2}z_2|^{6-4a_2}}\mathcal{N}_{\varphi_1,a_1}(z_1)\mathcal{N}_{\varphi_2,a_2}(z_2)dA(z_1)dA(z_2)
\end{multline}
Now, for $\omega_1,\omega_2\in\mathbb{D}$ with $|\omega_i|>1/2,i=1,2$ let
$$D(\omega_i)=\{z_i\in\mathbb{D}:|z_i-\omega_i|<\frac{1}{2}(1-|\omega_i|^2)\}.$$
By Aleman's inequality that one can find in \cite{Aleman} and by the well known fact that the quantities $|1-\overline{\omega_i}z_i|$ and $(1-|\omega_i|^2)$ are comparable for every $z_i\in D(\omega_i),i=1,2,$ we get the inequalities:
\begin{align}
  \mathcal{N}_{\varphi_i,a_i}(\omega_i)&\leq \frac{4}{(1-|\omega_i|^2)}\int_{D(\omega_i)}\mathcal{N}_{\varphi_i,a_i}(z_i)dA(z_i)&\\
 & \leq C_i\int_{\mathbb{D}}\frac{(1-|\omega_i|^2)^{3-2a_i}}{|1-\overline{\omega_i}z_i|^{6-4a_i}}\mathcal{N}_{\varphi_i,a_i}(z_i)dA(z_i)&\\
 & \leq C_i(1-|\omega_i|^2)^{1-2a_i},
\end{align}
hence, $$\sup_{|\omega_i|>1/2}\frac{\mathcal{N}_{\varphi_i,a_i}(\omega_i)}{(1-|\omega_i|^2)^{1-2a_i}}<\infty.$$
By the proof of Aleman's inequality, it is assured that there exist two subharmonic functions $u_{a_i},i=1,2$ such that $\mathcal{N}_{\varphi_i,a_i}\leq u_{a_i}.$ By boundedness of upper semi-continuous functions on compact sets, we obtain
$$\sup_{|\omega_i|\leq 1/2}\frac{\mathcal{N}_{\varphi_{i},a_i}(\omega_i)}{(1-|\omega_i|^2)^{1-2a_i}}\leq 2^{1-2a_i}\sup_{|\omega_i|\leq 1/2}u_{a_i}(\omega_i),i=1,2$$ and we are done with this direction of the proof as well.
\end{proof}
\begin{remark} The previous result can be extended in polydisc setting. Consider $\Phi:\mathbb{D}^n\to\mathbb{D}^n$ which is seperated, in the sense that $\Phi(z_1,...,z_n)=(\varphi_1(z_1),...,\varphi_n(z_n)).$ Then, $C_{\Phi}$ is bounded in the anisotropic Dirichlet space if and only if the condition of Theorem 2.1. is met for $i=1,...,n.$
\end{remark}
\section{Composition operators between the Dirichlet and the Bergman space on the bidisc .}
In this section we will consider the problem of the boundedness of the composition operator $C_{\Phi}:A^2_{\beta}(\mathbb{D}^2)\to \mathfrak{D}(\mathbb{D}^2)$. We have already stated that for the case of the Bergman space and the Hardy spaces of the bidisc, a significant ammount of work has been done in \cite{Bayart1} \cite{Bayart2} and \cite{Kosinski}. The motivation that drove us to consider the following problem is simple. The case of the boundedness of the composition operator in the Dirichlet space of the polydisc has insurmountable amount of calculations on the derivatives, and makes the pull-back measure approach quite difficult to handle. The problem of studying the boundedness of the composition operator $C_{\Phi}:A^2_{\beta}(\mathbb{D}^2)\to \mathfrak{D}(\mathbb{D}^2)$ seems to give more initial results, due to the existence of theorems such as those that we will see in the next two subsections. In the  first subsection we will derive a only sufficient condition which works for seperated symbols, while on the last subsection we will give a necessary and sufficient condition for all holomorphic symbols $\Phi:\mathbb{D}^2\to\mathbb{D}^2.$
 
\subsection{A sufficient condition.}
Our main tool for this section will be the Carleson measures for the Dirichlet space on the bidisc. We denote by $I\times J$ the product of two arcs in the torus $\mathbb{T}.$ Let $z=(z_1,z_2)\in\mathbb{D}^2,$ $\zeta=(\zeta_1,\zeta_2)\in\mathbb{T}^2$ and $\delta=(\delta_1,\delta_2)\in (0,1)^2.$ A Carleson box on the bidisc is the following set: $S(I\times J)=\{z\in\mathbb{D}^2:|z_i-\zeta_i|<\delta_i\},i=1,2.$ A positive Borel measure in $\mathbb{D}^2$ is called \textit{Carleson measure} on $\mathfrak{D}(\mathbb{D}^2)$ if the \textit{Carleson embedding} is a bounded operator, that is, if there exists a positive constant $C(\mu)>0$ such that $$\int_{\overline{\mathbb{D}^2}}|f|^2d\mu \leq C(\mu)||f||_{\mathfrak{D}(\mathbb{D}^2)}.$$ As we will observe, it is not easy to characterise the measures that satisfy the previous inequality. Our main target in this section is to provide a sufficient one-box condition for Carleson measures of the Dirichlet space on the bidisc which is inspired by the techniques of \cite{Fallah}. We will use this condition to obtain a one-box volume condition for the boundedness of the composition operator between the Bergman and the Dirichlet space. For what follows, $K_{w}(z)$ will denote the reproducing kernel of the Dirichlet spaces on the bidisc, that is:
$$K_{w}(z)=\left(C_1+\log\frac{2}{1-\overline{w_1}z_1}\right)\left(C_2+\log\frac{2}{1-\overline{w_2}z_2}\right)$$
for $z=(z_1,z_2), w=(w_1,w_2)\in\mathbb{D}^2.$ The main goal of this section is to prove 
\begin{theorem} \textit{Let $\Phi:\mathbb{D}^2\to\mathbb{D}^2$ be a holomorphic self map of the bidisc. Assume that there exist an increasing function $\psi$ with :
$$\int_{0}^{2\pi}\int_{0}^{2\pi}\frac{\psi(xy)}{xy}dxdy<\infty,$$
and $C>0$ such that:}
$$V_{\beta}\left(\Phi^{-1}(S(I\times J)\right)\leq C\psi(|I\times J|)$$
\textit{Then, the composition operator $C_{\Phi}:\mathfrak{D}(\mathbb{D}^2)\to A_{\beta}^2(\mathbb{D}^2)$ is bounded.}
\end{theorem}
For the proof of the above sufficient condition, we will establish the following lemmata:
\begin{lemma} \textit{Let $\mu$ a positive Borel measure on $\mathbb{D}^2$. Then, $\mu $ is a Carleson measure on $\mathfrak{D}(\mathbb{D}^2)$, if and only if}
$$\int_{\mathbb{D}^2} \int_{\mathbb{D}^2}\mathfrak{R}K_w(z)|g(z)||g(w)|d\mu(z)d\mu(w)\leq C||g||_{L^2(\mathbb{D}^2,d\mu)},$$
\textit{for any $g\in L^2(\mathbb{D}^2,d\mu). $}
\end{lemma}
\begin{proof} The proof can be found in \cite{Arcozzi2}. Also the above lemma holds for all Reproducing Kernel Hilbert Spaces defined in an arbitrary polydisc $\mathbb{D}^n,n>2.$ For more details the interested reader can check \cite{Arcozzi3}
\end{proof}
\begin{lemma}
\textit{Let $\mu$ a positive Borel measure on $\mathbb{D}^2$. If:}
$$\sup_{w\in \mathbb{D}^2}\int_{\mathbb{D}^2}\mathfrak{R}K_{w}(z)d\mu(z)<+\infty$$
\textit{then $\mu$ is a Carleson measure for $\mathfrak{D}(\mathbb{D}^2).$}
\end{lemma}
\begin{proof}
By \textit{Cauchy-Schwarz} inequality and the symmetry property of the kernel function, we obtain:
$$
\int_{\mathbb{D}^2}\int_{\mathbb{D}^2}\mathfrak{R}K_w(z)|g(z)||g(w)|d\mu(z)d\mu(w)\leq \int_{\mathbb{D}^2}\int_{\mathbb{D}^2}\mathfrak{R}K_z(w)|g(w)|^2d\mu(z)d\mu(w)$$
and the result follows immediately.
\end{proof}
Now we will state the two dimensional version of Theorem 1.1. of \cite{Fallah} which will be the key for the establishment of the sufficient condition that we stated at the starting point of this subsection. 
\begin{lemma} \textit{Let $\mu$ be two finite positive Borel measures on $\mathbb{D}^2.$ Let  $\psi$ be non decreasing functions with :
$$\int_{0}^{2\pi}\int_{0}^{2\pi}\frac{\psi(xy)}{xy}dxdy<\infty.$$ If} 
$$\mu(S(I\times J))\leq C \psi(|I\times J|)$$
\textit{for $I,J$ two arcs on $\mathbb{T},$ then, the measure $\mu$ is a Carleson measure for $\mathfrak{D}(\mathbb{D}^2).$}
\end{lemma}
\begin{proof}
 We will follow the idea of the proof of Theorem 1.1. on \cite{Fallah}. We will bound the integral of the real part of the reproducing kernel. For the measure $\mu$ it is true that:
 $$\int_{\mathbb{D}^2}\mathfrak{R}K_{w}(z)d\mu(z)=\int_{\mathbb{D}^2}\left(C_1+\log\frac{2}{|1-\overline{w_1}z_1|}\right)\left(C_2+\log\frac{2}{|1-\overline{w_2}z_2|}\right)d\mu(z_1,z_2).$$
 The right-hand side integral can be re-written now as follows:
 \begin{multline}
 \int_{\mathbb{D}^2}\left(C_1+\log\frac{2}{|1-\overline{w_1}z_1|}\right)\left(C_2+\log\frac{2}{|1-\overline{w_2}z_2|}\right)d\mu(z_1,z_2)=\\ \int_{\mathbb{D}^2}\left(\int_{|1-\overline{w_1}z_1|=t_1}^{2}\int_{|1-\overline{w_2}z_2|=t_2}^2\frac{1}{t_1}\frac{1}{t_2}dt_1dt_2\right)d\mu(z_1,z_2)\end{multline}
 Now, the product of the sets $\{|1-\overline{w_1}z_1|\leq t_1\}\times \{|1-\overline{w_2}z_2|\leq t_2\}$ is contained inside $S(I\times J)$ where $I\times J$ is the Cartesian product of two arcs, with $|I\times J|=64t_1t_2.$ So now it is easy to observe that
 $$\mu(\{(z_1,z_2)\in \mathbb{D}^2:|1-\overline{w_1}z_1|\leq t_1,|1-\overline{w_2}z_2|\leq t_2\})\leq C\psi(64t_1t_2).$$
 Hence,
 $$\int_{\mathbb{D}^2}\mathfrak{R}K_{w}(z)d\mu(z)\leq C\int_{0}^{8}\int_{0}^{8}\frac{\psi(t_1t_2)}{t_1t_2}dt_1dt_2<+\infty.$$
\end{proof}
At this point we are ready for the proof of Theorem 3.1.
\begin{proof} 
The composition operator $C_{\Phi}:\mathfrak{D}(\mathbb{D}^2)\to A^2_{\beta}(\mathbb{D}^2)$ is bounded if and only if:
$$\int_{\mathbb{D}^2}|f (z_1,z_2)|^2dV_{\beta}\circ \Phi^{-1}(z_1,z_2)\leq C\int_{\mathbb{D}^2}|\partial_{z_2}\partial_{z_1}[z_1z_2f(z_1,z_2)]|^2dV(z_1,z_2),$$
which is equivalent to the measure $V_{\beta}\circ\Phi^{-1}$ being a Carleson measure for $\mathfrak{D}(\mathbb{D}^2)$ By our assumption and Lemma 3.3 this is the case, and the proof is now completed
\end{proof} 
\begin{remark} The above arguments and lemmata work for any polydisc $\mathbb{D}^n,n\ge 2.$ One can find this fact on \cite{Mozolyako}.
\end{remark}
\subsection{A necessary and sufficient condition} In this section we give a necessary and sufficient condition for the boundedness of $C_{\Phi}:\mathfrak{D}(\mathbb{D}^2)\to A^2_{\beta}(\mathbb{D}^2)$. In order to do that we have to give some backround Theorems.  For what follows, we assume that $E$ is a compact subset of the bitorus $\mathbb{T}^2.$ The \textit{Bessel capacity of $E$} with parameter $1/2$ is defined in the following manner. Consider $\theta=(\theta_1,\theta_2)$ and $\eta=(\eta_1,\eta_2)$ such that $\theta_1,\theta_2,\eta_1,\eta_2\in [-\pi,\pi).$ We consider the 2-dimensional extension of the  $1/2$-Bessel kernel on the bi-torus as
$$k_{\mathbb{T}^2,1/2}(\theta,\eta)=|\theta_1-\eta_1|^{-1/2}|\theta_2-\eta_2|^{-1/2}.$$ The above kernel can be extended to a convolution operator as follows:
$$k_{\mathbb{T}^2,1/2}\psi(\theta)=\int_{\mathbb{T}^2}k_{\mathbb{T}^2,1/2}(\theta-\eta)d\psi(\eta).$$
Now let $E\subset \mathbb{T}^2$ be a closed subset of the bi-torus. The Bessel $1/2$-capacity of $E$ is:
$$\mathrm{Cap}_{\mathbb{D}^2}^{1/2}(E)=\inf \{||h||^2_{L^2(\mathbb{T}^2)}:h\ge 0,k_{\mathbb{T}^2,1/2}h\ge 1, on E\}$$
In their seminal work \cite{Arcozzi2} , Arcozzi, Mozolyako, Perfekt and Sarfatti proved the following theorem for the classical Dirichlet space on the bi-disc which characterises the Carleson measures in this space, generalising the result of Stegenga on the disc setting. 
\begin{theorem} \textit{Let $\mu\ge 0$ be a Borel measure on $\overline{\mathbb{D}^2}.$ Then the following are equivalent:}
\begin{enumerate}
\item \textit{There exists a constant $C_1>0$ such that:}
 $$\int_{\overline{\mathbb{D}^2}}|f|^2d\mu \leq C_1||f||_{\mathfrak{D}(\mathbb{D}^2)}.$$
\item \textit{There exists a constant $C_2>0$ such that $\forall n \ge 1$ and for all choices of arcs $J_1^1,...,J_n^1, J_1^2,...,J_n^2$ on $\mathbb{T}$ we have that:}
$$\mu\left(\bigcup_{k=1}^nS(J_k^1\times J_k^2)\right)\leq C_2 \mathrm{Cap}_{\mathbb{D}^2}^{1/2}\left(\bigcup_{k=1}^n J_k^1\times J_k^2\right).$$
\textit{The constants $C_1,C_2$ are comparable independently of $\mu$.}
\end{enumerate}
\end{theorem}
We remind that $S(J_k^1\times J_k^2)$ denotes the cartesian product of Carleson boxes associated to the arcs $J^1_k\times J^2_k,$ generalising the concept of the usual one-dimensional Carleson boxes on the bidisc case.
Utilising the previously mentioned tools, we can obtain a necessary and sufficient condition for the boundedness of $C_{\Phi}: \mathfrak{D}(\mathbb{D}^2)\to A^2_{\beta}(\mathbb{D}^2).$ 
We obtain the following
\begin{theorem} \textit{Let $\Phi:\mathbb{D}^2\to \mathbb{D}^2$ holomorphic. Then $C_{\Phi}:\mathfrak{D}(\mathbb{D}^2)\to A^2_{\beta}(\mathbb{D}^2)$ for $\beta>-1,$ is bounded, if and only if there exists a constant $C>0$ such that $\forall n \ge 1$ and for all choices of arcs $I_1^1,...,I_n^1, I_1^2,...,I_n^2$ on $\mathbb{T}$ we have that:}
$$V_{\beta}\left(\Phi^{-1}\left(\bigcup_{k=1}^n S(I_k^1\times I_k^2)\right)\right)\leq C \mathrm{Cap}_{\mathbb{D}^2}^{1/2}\left(\bigcup_{k=1}^n I_k^1\times I_k^2\right).$$ 
\end{theorem}
\begin{proof}
The composition operator $C_{\Phi}:\mathfrak{D}(\mathbb{D}^2)\to A^2_{\beta}(\mathbb{D}^2)$ is bounded if and only if
$$\int_{\mathbb{D}^2}|f\circ \Phi(z_1,z_2)|^2dV_{\beta}(z_1,z_2)\leq C\int_{\mathbb{D}^2}|\partial_{z_2}\partial_{z_1}[z_1z_2f(z_1,z_2)]|^2dV(z_1,z_2).$$
By taking the pull-back measure this condition turns into the following:
$$\int_{\mathbb{D}^2}|f (z_1,z_2)|^2dV_{\beta}\circ \Phi^{-1}(z_1,z_2)\leq C\int_{\mathbb{D}^2}|\partial_{z_2}\partial_{z_1}[z_1z_2f(z_1,z_2)]|^2dV(z_1,z_2).$$
But observing the above condition, this is equivalent to the measure $d\mu=dV_{\beta}\circ \Phi^{-1}$ being a Carleson measure for the Dirichlet space on the bi-disc. So by applying Theorem 3.2. we obtain 
that the composition operator $C_{\Phi}:\mathfrak{D}(\mathbb{D}^2)\to A^2_{\beta}(\mathbb{D}^2)$ is bounded if and only if
$$V_{\beta}\left(\Phi^{-1}\left(\bigcup_{k=1}^n S(I_k^1\times I_k^2)\right)\right)\leq C \mathrm{Cap}_{\mathbb{D}^2}^{1/2}\left(\bigcup_{k=1}^n I_k^1\times I_k^2\right)$$ 
for all choices of arcs $I_1^1,...,I_n^1, I_1^2,...,I_n^2$ on $\mathbb{T}$, for all $n\ge 1.$
\end{proof}
\begin{remark} In the inequality above we can replace the 1/2-Bessel capacity with the logarithmic capacity generated by the kernel function:
$$K_{w}(z)=\left(C_1+\log\frac{2}{1-\overline{w_1}z_1}\right)\left(C_2+\log\frac{2}{1-\overline{w_2}z_2}\right)$$
for $z=(z_1,z_2), w=(w_1,w_2)\in\mathbb{D}^2.$
This fact is pointed out in \cite{Arcozzi2}. 
\end{remark}
\begin{remark} The capacity of the the product of arcs can be replaced by the capacity of the corresponding 2-dimensional Carleson box 
\end{remark}
\section{Further discussion} Observing the work that has been done for the case of Bergman and Hardy spaces on the bidisc by Bayart and Kosinski in \cite{Bayart1}, \cite{Bayart2}, \cite{Kosinski} we can say that these types of volume conditions can help us provide characterizations of the symbols that induce bounded composition operators between Bergman and Dirichlet spaces. It would be more pleasant to work only on the Dirichlet-type space on the bidisc, but as we mentioned before, the involvement of the derivative is what makes things much more complicated. The necessary and sufficient condition we provide here for the boundedness of $C_{\Phi}:\mathfrak{D}(\mathbb{D}^2)\to A^2_{\beta}(\mathbb{D}^2)$ is still quite difficult to work with because it does not involve only one Carleson box but rather a finite union of them. The sufficient condition on the other hand has the positive aspect of being stated only for one Carleson box. Another interesting observation we should point out here is that Theorem 3.3. is true only for the bidisc, as the analogue of Theorem 3.2. is not proved for the poly-disc but only for the bidisc setting. Last but not least, for the case of the anisotropic Dirichlet-type spaces, we sadly do not yet have the extension of Theorem 3.2. for all $a_i<1/2,$ and this is the main reason that Theorem 3.2. is only stated for the unweighted Dirichlet space. If Theorem 3.2. happens to be true for some set of parameters $a_i$ then we can characterize the composition operators that are bounded between the Bergman and Dirichlet-type spaces for the anisotropic setting for the said set of parameters.\\
Another interesting topic to be discussed, is the extension of the above techniques to other domains. If a domain  $\Omega\subset\mathbb{C}^n$ is a product domain, then the separated symbols case seems to provide similar results as the ones obtained in our work, but in Bergman spaces of the said domain $\Omega.$ 
\section{Acknowledgements:} The author would like to thank Prof. N.Chalmoukis and Prof. K-M Perfekt for the helpful correspondence during the preparation of this note. I would also like to thank the anonymous referee for the constructive suggestions that improved the readability of the paper significantly.
\section{Financial support}
The author was partially supported by the National Science Center, Poland, SHENG III, research project 2023/48/Q/ST1/00048\\\\

Athanasios Beslikas,\\
Doctoral School of Exact and Natural Studies\\
Institute of Mathematics,\\
Faculty of Mathematics and Computer Science,\\
Jagiellonian University\\ 
\L{}ojasiewicza 6\\
PL30348, Cracow, Poland\\
athanasios.beslikas@doctoral.uj.edu.pl


\begin{thebibliography}{8}
\bibitem{Aleman}  A. Aleman, \textit{Hilbert spaces of analytic functions between the Hardy and the Dirichlet space}, Proc. Amer. Math. Soc. 115 (1992) 97–104.
\bibitem{Arcozzi1} N. Arcozzi, R. Rochberg, E. T. Sawyer
and B. D. Wick \textit{The Dirichlet space: a survey}
New York J. Math. 17a (2011) 45–86.
\bibitem{Arcozzi2} N. Arcozzi P. Mozolyako K-M. Perfekt
G. Sarfatti, \textit{Bi-parameter Potential Theory and Carleson
Measures for the Dirichlet Space on the Bidisc}
Discrete Analysis, 2023:22, 57 pp.
\bibitem{Arcozzi3} N. Arcozzi,
N. Chalmoukis, M. Levi and P. Mozolyako, \textit{Two-weight dyadic Hardy inequalities,}  Rend. Lincei Mat. Appl. 34 (2023), 657–714
DOI 10.4171/RLM/1023
\bibitem{Balooch} A.Balooch, Z.Wu, \textit{A new formalization of Dirichlet-type spaces} Journal of Mathematical Analysis and Applications
Volume 526, Issue 1, 1 October 2023, 127322.
\bibitem{Bayart1} F.Bayart, \textit{Composition operators on the Hardy Space of the tridisc}, arXiv:2312.02565v1.
\bibitem{Bayart2} F.Bayart, \textit{Composition operators on the polydisk induced by affine maps} Journal of Functional Analysis
Volume 260, Issue 7, 1 April 2011, Pages 1969-2003.
\bibitem{ransf} O. El-Fallah, K. Kellay, J. Mashreghi, and T. Ransford, A primer on the Dirichlet space, Cambridge Tracts in Mathematics 203, Cambridge University Press, 2014.
\bibitem{Fallah} O. El-Fallah, K. Kellay, J. Mashreghi, and T. J. Ransford \textit{One-Box conditions for Carleson Measures
For The Dirichlet Space.,} Proceedings of the American Mathematical Society,
Volume 143, Number 2, February 2015, Pages 679–684
S 0002-9939(2014)12248-9
\bibitem{Kaptanoglu} H.T.Kaptanoglu,  \textit{Mobius-invariant Hilbert spaces in Polydiscs }Pacific Journal Of Mathematics
Vol. 163, No. 2, 1994.
\bibitem{Kosinski2} G. Knese, \L{}. Kosinski, T. J. Ransford, A. Sola \textit{Cyclic polynomials in anisotropic Dirichlet spaces}
JAMA 138, 23–47 (2019). https://doi.org/10.1007/s11854-019-0014-x
\bibitem{Kosinski} \L{}.Kosinski, \textit{Composition operators on the polydisc}, Journal of Functional Analysis,
Volume 284, Issue 5, 1 March 2023, 109801.
\bibitem{Mozolyako} Pavel Mozolyako, Georgios Psaromiligkos, Alexander Volberg,
and Pavel Zorin-Kranich,\textit{Carleson embedding on tri-tree and on tri-disc,} 
https://doi.org/10.48550/arXiv.2001.02373
\bibitem{Pau} J.Pau, P.A.Perez, \textit{Composition operators acting on weighted Dirichlet spaces}, Journal of Mathematical Analysis and Applications
Volume 401, Issue 2, 15 May 2013, Pages 682-694.
\bibitem{Shapiro} J.H.Shapiro \textit{The essential norm of a composition operator,} Ann. Math. 12 (1987), 375-404.
\bibitem{Stegenga} D.Stegenga, \textit{Multipliers of the Dirichlet space} Illinois Journal Of Mathematics
Volume 24, Number 1, Spring 1980.
\end{thebibliography}
\end{document}